\documentclass[11pt]{amsart}

\usepackage[T1]{fontenc}
\usepackage[margin=1in]{geometry}
\usepackage{lmodern}
\usepackage{microtype}
\usepackage{amsmath,amssymb,amsthm,mathtools}
\usepackage{enumitem}
\usepackage{aliascnt}
\usepackage[hidelinks]{hyperref}
\usepackage[nameinlink,capitalize,noabbrev]{cleveref}

\newcommand{\E}{\mathbb{E}}
\newcommand{\Pp}{\mathbb{P}}

\newcommand{\R}{\mathbb{R}}
\newcommand{\1}{\mathbf{1}}
\newcommand{\cM}{\mathcal{M}}
\newcommand{\cF}{\mathcal{F}}
\newcommand{\cH}{\mathcal{H}}
\newcommand{\MC}{\operatorname{MaxCut}}
\newcommand{\MB}{\operatorname{MaxBis}}
\newcommand{\cut}{\operatorname{cut}}
\newcommand{\CM}{\widehat G}
\newcommand{\dd}{\mathbf d}

\DeclareMathOperator{\Var}{Var}

\theoremstyle{plain}
\newtheorem{theorem}{Theorem}[section]
\newaliascnt{proposition}{theorem}
\newtheorem{proposition}[proposition]{Proposition}
\aliascntresetthe{proposition}
\newaliascnt{lemma}{theorem}
\newtheorem{lemma}[lemma]{Lemma}
\aliascntresetthe{lemma}
\newaliascnt{corollary}{theorem}
\newtheorem{corollary}[corollary]{Corollary}
\aliascntresetthe{corollary}

\theoremstyle{definition}
\newaliascnt{definition}{theorem}

\aliascntresetthe{definition}

\theoremstyle{remark}
\newaliascnt{remark}{theorem}

\aliascntresetthe{remark}

\numberwithin{equation}{section}

\title[Maximum cut and maximum bisection]{Maximum cut and maximum bisection in random regular graphs}
\author{P. M. Aronow and Patrick Lopatto}
\date{July 19, 2026}


\begin{document}

\begin{abstract}
We prove that, for every fixed degree $d$ and a uniformly random simple $d$-regular graph $G_{n,d}$ on an even number $n>d$ of vertices,
\[
  \E\MC(G_{n,d})-\E\MB(G_{n,d})=o(n).
\]
In other words, requiring the two sides of a cut to have exactly the same size changes the expected optimal cut by only a sublinear number of edges. We prove this first for the configuration model by comparing cuts of each possible cardinality on an $n$-vertex graph with bisections of a related graph on $2n$ vertices. An important technical tool is Huang's interpolation theorem \cite{Huang2018}; to apply it, we establish a structural property of the change in the optimal cut when a single edge is added. This comparison, together with concentration estimates and Huang's convergence theorem for maximum bisection, shows that maximum cut and maximum bisection have the same limiting density. Conditioning the configuration model on being simple then gives the result for uniformly random simple regular graphs.
\end{abstract}

\maketitle

\section{Introduction}

For a finite undirected multigraph $G=(V,E)$ and a set $S\subseteq V$, let
\[
  \cut_G(S)=e_G(S,V\setminus S)
\]
be the number of edges with exactly one endpoint in $S$. Multiple edges are counted with multiplicity, while loops make no contribution. Define
\[
  \MC(G)=\max_{S\subseteq V}\cut_G(S)
\]
and
\[
  \MB(G)=\max_{\substack{S\subseteq V\\ |S|=\lfloor |V|/2\rfloor}}\cut_G(S).
\]
Taking complements shows that the cardinality constraint in the second display may equivalently be $|S|=\lceil |V|/2\rceil$. In particular, for even $|V|$ this is the usual exact bisection. By definition, one always has $\MB(G)\leq\MC(G)$.

For fixed $d$, let $G_{n,d}$ be uniformly distributed over the simple $d$-regular graphs on $[n]$ whenever $n$ is even and $n>d$. Bandeira asked whether the expected maximum cut and expected maximum bisection differ by a sublinear term; this is Conjecture~33 in \cite{Bandeira2026}. The question is distinct from the related conjecture of Zdeborov\'a and Boettcher \cite{ZdeborovaBoettcher2010}, which compares maximum cut with minimum bisection. Dembo, Montanari, and Sen obtained matching large-degree asymptotics for maximum cut and maximum bisection \cite{DemboMontanariSen2017}, but those asymptotics do not settle the problem at any fixed value of $d$.

Our contribution is the following theorem.

\begin{theorem}\label{thm:main}
For every fixed integer $d\geq1$, as $n\to\infty$ through even integers $n>d$,
\begin{equation*}
  \E\MC(G_{n,d})-\E\MB(G_{n,d})=o(n).
\end{equation*}
Moreover, there is a constant $b_d$ such that
\[
  \frac1n\E\MC(G_{n,d})\longrightarrow b_d,
  \qquad
  \frac1n\E\MB(G_{n,d})\longrightarrow b_d.
\]
\end{theorem}

The proof is based on Salez's prescribed-degree interpolation framework \cite{Salez2016}, as extended by Huang to non-additive bisection parameters \cite{Huang2018}. We use Huang's interpolation theorem as a black box and verify its two hypotheses for the constrained cut functional needed here.

\subsection{Proof sketch}
The proof has three stages. We first work in the $d$-regular configuration model, denoted by $\CM_{n,d}$. For $0\leq k\leq n$, set
\[
  X_{n,k}(G)=\max_{\substack{S\subseteq[n]\\ |S|=k}}\cut_G(S),
  \qquad
  a_{n,k}=\E X_{n,k}(\CM_{n,d}).
\]
The main finite-size estimate is
\begin{equation}\label{eq:intro-key}
  2a_{n,k}
  \leq
  \E\MB(\CM_{2n,d})
  +7\sqrt{dn\log(1+dn)},
  \qquad 0\leq k\leq n.
\end{equation}
To obtain it, we split the $2n$ vertices into two blocks $A$ and $B$ of size $n$ and maximize the cut over sets satisfying
\[
  |S\cap A|=k,
  \qquad
  |S\cap B|=n-k.
\]
At the disconnected endpoint of the interpolation, the two blocks are independent copies of $\CM_{n,d}$, and the expected constrained optimum is $2a_{n,k}$. On the unconditioned complete-pairing side of Huang's comparison, the constrained family consists entirely of sets of size $n$. Hence its optimum is bounded above, pointwise, by $\MB(\CM_{2n,d})$.

The key local input is a structural property of maxima over fixed cut families. Let $\cF$ be any predetermined nonempty family of subsets of a finite vertex set and define
\[
  g_{\cF}(G)=\max_{S\in\cF}\cut_G(S).
\]
For vertices $u,v$, let
\[
  \Delta^G_{uv}=g_{\cF}(G+uv)-g_{\cF}(G).
\]
Two vertices have increment zero exactly when every current optimizer places them on the same side of the cut. Thus $\Delta^G$ is the $0$--$1$ separation matrix of a partition. On the subspace $\sum_v z_v=0$, its quadratic form is a negative sum of squares, and hence
\[
  \sum_{u,v}\Delta^G_{uv}z_uz_v\leq0.
\]
A weighted use of this inequality for the remaining half-edges gives Huang's local interpolation condition and therefore \eqref{eq:intro-key}.

In the second stage, each variable $X_{n,k}(\CM_{n,d})$ satisfies the same sub-Gaussian concentration bound on the scale $\sqrt n$. A union bound over the $n+1$ possible cardinalities controls the maximum deviation on the scale $\sqrt{n\log n}$ and, together with \eqref{eq:intro-key}, gives
\[
  \E\MC(\CM_{n,d})
  \leq
  \frac12\E\MB(\CM_{2n,d})
  +O_d(\sqrt{n\log n}).
\]
Huang proved that $n^{-1}\E\MB(\CM_{n,d})$ converges. The preceding inequality, together with the pointwise bound $\MC\geq\MB$, forces the normalized expected maximum cut to converge to the same limit.

Finally, for fixed $d$, the configuration model is simple with probability bounded away from zero \cite{Janson2014}. Concentration shows that conditioning on simplicity changes either expectation by at most $O_d(\sqrt n)$. Since the conditioned configuration model is uniform over simple $d$-regular graphs, this proves \Cref{thm:main}.

The new ingredient is the fixed-family increment lemma, \Cref{lem:increment}. It permits unequal cardinality constraints inside the two interpolation blocks and converts the local comparison required by Huang's theorem into a transparent partition identity.

\subsection{Organization}
\Cref{sec:configuration} introduces partial pairings and states the precise consequence of Huang's interpolation theorem used below. \Cref{sec:fixed-family} proves the increment-matrix lemma and verifies the two interpolation hypotheses for maxima over arbitrary predetermined cut families. \Cref{sec:doubling} applies this result to the two-block cardinality constraint and proves the doubling estimate \eqref{eq:intro-key}. \Cref{sec:configuration-proof} establishes concentration, invokes the maximum-bisection limit, and proves the common limiting density in the configuration model. Finally, \Cref{sec:simple} conditions on simplicity and completes the proof of \Cref{thm:main}.

\subsection{Acknowledgments}
P.L. was partially supported by NSF grant DMS-2450004. This paper was written with the assistance of large language models, which included suggesting arguments, drafting and revising the manuscript, and exploratory computational checks.

\section{Configuration model and interpolation criteria}\label{sec:configuration}

We begin by recalling the configuration model in a form adapted to interpolation. Let $V$ be a finite vertex set, and let $\dd=(d_v)_{v\in V}$ be a sequence of nonnegative integers with even total degree
\[
  D=\sum_{v\in V}d_v.
\]
Attach $d_v$ labeled half-edges to each $v$, and denote the set of all half-edges by $\cH_{\dd}$. A partial pairing is a matching of a subset of $\cH_{\dd}$. Each partial pairing $m$ induces a multigraph $G[m]$ on $V$: every paired pair of half-edges gives an edge between their incident vertices, and unmatched half-edges are ignored. A uniformly random complete pairing induces the configuration-model multigraph $\CM_{\dd}$. When $V=[n]$ and $d_v=d$ for every $v$, we write $\CM_{n,d}$.

Fix a partition $V=A\sqcup B$. An edge of a partial pairing is an $A$-edge if both half-edges are incident to vertices of $A$, a $B$-edge if both are incident to vertices of $B$, and a cross-edge otherwise. For $\alpha,\beta,\gamma\in\mathbb Z_{\geq0}$, let $\cM(\alpha,\beta,\gamma)$ be the set of partial pairings with exactly $\alpha$ $A$-edges, $\beta$ $B$-edges, and $\gamma$ cross-edges. We call the triple feasible when this set is nonempty.

A graph pseudo-parameter is a real-valued function on labeled multigraphs; invariance under relabeling is not required. For a pseudo-parameter $g$ and a feasible triple, set
\begin{equation*}
  F_g(\alpha,\beta,\gamma)
  =\E_{m\in\cM(\alpha,\beta,\gamma)}g(G[m]),
\end{equation*}
where the expectation is uniform. Write
\[
  D_A=\sum_{v\in A}d_v,
  \qquad
  D_B=\sum_{v\in B}d_v,
  \qquad
  \psi(x)=7\sqrt{x\log(1+x)}.
\]
Since $D_A+D_B=D$ is even, $D_A$ and $D_B$ have the same parity. Thus
\[
  \left(\left\lfloor\frac{D_A}{2}\right\rfloor,
  \left\lfloor\frac{D_B}{2}\right\rfloor,0\right)
\]
is feasible: pair all but at most one half-edge internally in each block.

The following proposition is the form of Huang's interpolation theorem needed below; it follows from \cite[Proposition~4.2 and Corollary~4.3]{Huang2018}.

\begin{proposition}\label{prop:huang}
Suppose that a graph pseudo-parameter $g$ satisfies the following two conditions.

\begin{enumerate}[label=\textup{(H\arabic*)},leftmargin=2.8em]
\item\label{item:H1}
For any feasible triples $(\alpha,\beta,\gamma)$ and $(\alpha',\beta',\gamma')$,
\begin{equation}\label{eq:H1}
  \bigl|F_g(\alpha,\beta,\gamma)-F_g(\alpha',\beta',\gamma')\bigr|
  \leq
  |\alpha-\alpha'|+|\beta-\beta'|+|\gamma-\gamma'|.
\end{equation}

\item\label{item:H2}
If $\delta\geq2$ and $(\alpha,\beta,\gamma+\delta)$ is feasible, then
\begin{equation}\label{eq:H2}
  \frac12\bigl(F_g(\alpha+1,\beta,\gamma)
  +F_g(\alpha,\beta+1,\gamma)\bigr)
  \leq
  F_g(\alpha,\beta,\gamma+1)+\frac2\delta.
\end{equation}
\end{enumerate}
Then
\begin{equation}\label{eq:huang-cor}
  F_g\left(\left\lfloor\frac{D_A}{2}\right\rfloor,
  \left\lfloor\frac{D_B}{2}\right\rfloor,0\right)
  \leq
  \E g(\CM_{\dd})+\psi(D/2).
\end{equation}
\end{proposition}

Huang works with finite undirected multigraphs, allowing loops and multiple edges, and formulates the interpolation theorem for arbitrary pseudo-parameters on labeled multigraphs; invariance under relabeling is not required. In particular, a pseudo-parameter may depend on a predetermined family of labeled vertex subsets, as well as on the fixed labeled partition $A\sqcup B$. Therefore \Cref{prop:huang} applies to every parameter $g_{\cF}$ considered below.

\section{Edge increments for fixed cut families}\label{sec:fixed-family}

This section verifies the hypotheses of \Cref{prop:huang} for a broad class of constrained cut parameters. The first step identifies the edge-increment matrix with the separation matrix of a partition determined by the current optimizers. The second step combines the resulting conditional negative semidefiniteness with the uniform extension property of partial pairings.

Let $V$ be finite, and let $\cF\subseteq2^V$ be a predetermined nonempty family. Define
\begin{equation}\label{eq:fixed-family-param}
  g_{\cF}(G)=\max_{S\in\cF}\cut_G(S).
\end{equation}
The family $\cF$ is fixed independently of $G$ and need not be closed under complements.

\subsection{The increment partition}

For a multigraph $G$ on $V$ and vertices $u,v\in V$, let $G+uv$ be the multigraph obtained by adding one edge between $u$ and $v$; when $u=v$, the added edge is a loop. Define
\begin{equation}\label{eq:increment-matrix}
  \Delta^G_{uv}=g_{\cF}(G+uv)-g_{\cF}(G).
\end{equation}

\begin{lemma}\label{lem:increment}
For every $G$, one has $\Delta^G_{uv}\in\{0,1\}$ and $\Delta^G_{uu}=0$. Moreover, the symmetric matrix $\Delta^G=(\Delta^G_{uv})_{u,v\in V}$ is conditionally negative semidefinite: for every $z\in\R^V$ satisfying $\sum_{v\in V}z_v=0$,
\begin{equation}\label{eq:cnd}
  \sum_{u,v\in V}\Delta^G_{uv}z_uz_v\leq0.
\end{equation}
\end{lemma}

\begin{proof}
Let
\[
  M=g_{\cF}(G),
  \qquad
  \mathcal O=\{S\in\cF:\cut_G(S)=M\}
\]
be the nonempty family of current optimizers. Define a relation on $V$ by
\begin{equation*}
  u\sim v
  \quad\Longleftrightarrow\quad
  \1_{\{u\in S\}}=\1_{\{v\in S\}}
  \text{ for every }S\in\mathcal O.
\end{equation*}
This is an equivalence relation. We first show that
\begin{equation}\label{eq:increment-equivalence}
  \Delta^G_{uv}=\1_{\{u\not\sim v\}}.
\end{equation}

Suppose that $u\not\sim v$. Some optimizer $S\in\mathcal O$ separates $u$ and $v$, and therefore
\[
  g_{\cF}(G+uv)\geq\cut_{G+uv}(S)=M+1.
\]
Adding a single edge increases the value of every feasible cut by at most one, so $g_{\cF}(G+uv)\leq M+1$. Hence $\Delta^G_{uv}=1$.

Suppose instead that $u\sim v$. No member of $\mathcal O$ separates the two vertices, so every current optimizer retains value $M$ after the edge is added. If $T\in\cF\setminus\mathcal O$, integrality gives $\cut_G(T)\leq M-1$, and consequently $\cut_{G+uv}(T)\leq M$. Thus the new optimum is still $M$, which proves \eqref{eq:increment-equivalence}. In particular, $u\sim u$, and hence $\Delta^G_{uu}=0$.

Let $C_1,\dots,C_r$ be the equivalence classes of $\sim$. By \eqref{eq:increment-equivalence},
\[
  \Delta^G_{uv}
  =1-\sum_{j=1}^r
  \1_{\{u\in C_j\}}\1_{\{v\in C_j\}}.
\]
Therefore, whenever $\sum_v z_v=0$,
\begin{align*}
  \sum_{u,v\in V}\Delta^G_{uv}z_uz_v
  &=\left(\sum_{v\in V}z_v\right)^2
    -\sum_{j=1}^r\left(\sum_{v\in C_j}z_v\right)^2\\
  &=-\sum_{j=1}^r\left(\sum_{v\in C_j}z_v\right)^2
  \leq0.
\end{align*}
This proves \eqref{eq:cnd}.
\end{proof}

\subsection{Verification of the interpolation hypotheses}

We first record the elementary extension property that permits us to express adjacent values of $F_g$ by conditional edge increments.

\begin{lemma}\label{lem:uniform-extension}
Suppose that both $(\alpha,\beta,\gamma)$ and $(\alpha+1,\beta,\gamma)$ are feasible. Sample $m$ uniformly from $\cM(\alpha,\beta,\gamma)$ and then add a uniformly chosen pair of distinct unmatched half-edges incident to $A$. The resulting partial pairing is uniform on $\cM(\alpha+1,\beta,\gamma)$. The analogous statements hold for adding a $B$-edge or a cross-edge.
\end{lemma}

\begin{proof}
Every $m\in\cM(\alpha,\beta,\gamma)$ has
\[
  a=D_A-2\alpha-\gamma
\]
unmatched half-edges incident to $A$, and hence exactly $\binom a2$ possible $A$-edge extensions. Conversely, every $m'\in\cM(\alpha+1,\beta,\gamma)$ has exactly $\alpha+1$ preimages, obtained by deleting one of its $A$-edges. Both numbers depend only on the endpoint triple. The extension procedure therefore assigns the same probability to every $m'$. The proofs for $B$-edges and cross-edges are identical.
\end{proof}

\begin{proposition}\label{prop:fixed-family-huang}
For every nonempty predetermined family $\cF\subseteq2^V$, the pseudo-parameter $g_{\cF}$ in \eqref{eq:fixed-family-param} satisfies \ref{item:H1} and \ref{item:H2} of \Cref{prop:huang}, for every partition $V=A\sqcup B$ and every degree sequence $\dd$.
\end{proposition}

\begin{proof}
We first verify \ref{item:H1}. Adding one edge changes $g_{\cF}$ by at most one. By \Cref{lem:uniform-extension}, if two feasible triples differ by one in a single coordinate, then the corresponding values of $F_{g_{\cF}}$ differ by at most one.

The set of feasible triples is coordinatewise downward closed: deleting selected edges from a witnessing partial pairing preserves feasibility. Given two feasible triples $q$ and $q'$, their coordinatewise minimum $q\wedge q'$ is therefore feasible. Every monotone coordinate path from $q\wedge q'$ to either $q$ or $q'$ remains inside the feasible set. Applying the triangle inequality through $q\wedge q'$ and summing the adjacent one-edge bounds along the two paths gives \eqref{eq:H1}.

It remains to prove \ref{item:H2}. Fix $\delta\geq2$ such that $(\alpha,\beta,\gamma+\delta)$ is feasible. Deleting $\delta$ cross-edges from a witnessing matching shows that $(\alpha,\beta,\gamma)$ is feasible; fix $m\in\cM(\alpha,\beta,\gamma)$. Write $G=G[m]$. For $v\in V$, let $c_v$ be the number of unmatched half-edges incident to $v$, and set
\begin{equation*}
  a=\sum_{v\in A}c_v=D_A-2\alpha-\gamma,
  \qquad
  b=\sum_{v\in B}c_v=D_B-2\beta-\gamma.
\end{equation*}
The assumed feasibility of $(\alpha,\beta,\gamma+\delta)$ implies
\begin{equation}\label{eq:ab-delta}
  a\geq\delta,
  \qquad
  b\geq\delta.
\end{equation}
In particular, the three neighboring triples appearing in \eqref{eq:H2} are feasible.

Conditional on $m$, we define an available $A$-edge to mean an unordered pair of distinct unmatched half-edges both incident to vertices of $A$; available $B$-edges are defined analogously, and an available cross-edge means a pair consisting of one unmatched half-edge incident to $A$ and one incident to $B$. Let $e_A$, $e_B$, and $e_{AB}$ be chosen uniformly from these three respective sets of available half-edge pairs, and identify each chosen pair with the edge between the vertices incident to its two half-edges. Define
\[
  I_A=\E\bigl[g_{\cF}(G+e_A)-g_{\cF}(G)\mid m\bigr],
\]
and define $I_B$ and $I_{AB}$ analogously.

Let $\Delta=\Delta^G$ be the increment matrix in \eqref{eq:increment-matrix}. Introduce the probability weights
\[
  \mu_A(u)=\frac{c_u}{a}\quad (u\in A),
  \qquad
  \mu_B(v)=\frac{c_v}{b}\quad (v\in B),
\]
and set
\begin{align*}
  Q_A&=\sum_{u,v\in A}\mu_A(u)\mu_A(v)\Delta_{uv},\\
  Q_B&=\sum_{u,v\in B}\mu_B(u)\mu_B(v)\Delta_{uv},\\
  Q_{AB}&=\sum_{\substack{u\in A\\v\in B}}
  \mu_A(u)\mu_B(v)\Delta_{uv}.
\end{align*}
Since $0\leq\Delta_{uv}\leq1$, one has $0\leq Q_A,Q_B\leq1$.

Because the summand $\Delta_{uv}$ is symmetric in $u$ and $v$, choosing a uniformly random unordered pair of distinct unmatched half-edges incident to $A$ gives the same expectation as choosing a uniformly random ordered pair of such half-edges. If both half-edges are incident to the same vertex, their contribution vanishes because $\Delta_{uu}=0$. Consequently,
\begin{equation}\label{eq:I-vs-Q}
  I_A=\frac{a}{a-1}Q_A,
  \qquad
  I_B=\frac{b}{b-1}Q_B,
  \qquad
  I_{AB}=Q_{AB}.
\end{equation}
The first identity follows from
\begin{align*}
  I_A
  &=\frac1{a(a-1)}
    \sum_{u,v\in A}
    \bigl(c_uc_v-\1_{\{u=v\}}c_u\bigr)\Delta_{uv}\\
  &=\frac1{a(a-1)}\sum_{u,v\in A}c_uc_v\Delta_{uv}
   =\frac{a}{a-1}Q_A.
\end{align*}
The other identities are proved similarly.

Apply \Cref{lem:increment} to the vector $z\in\R^V$ defined by
\[
  z_u=\mu_A(u)\quad (u\in A),
  \qquad
  z_v=-\mu_B(v)\quad (v\in B).
\]
Its coordinates sum to zero, and hence
\begin{equation}\label{eq:Q-cnd}
  Q_A+Q_B-2Q_{AB}\leq0.
\end{equation}
Combining \eqref{eq:I-vs-Q} and \eqref{eq:Q-cnd} gives
\begin{align}
  \frac{I_A+I_B}{2}-I_{AB}
  &=\frac12\bigl(Q_A+Q_B-2Q_{AB}\bigr)
    +\frac{Q_A}{2(a-1)}+\frac{Q_B}{2(b-1)}\notag\\
  &\leq \frac1{2(a-1)}+\frac1{2(b-1)}\notag\\
  &\leq \frac1{\delta-1}
  \leq \frac2\delta.
  \label{eq:local-increment}
\end{align}
Here the penultimate inequality uses \eqref{eq:ab-delta}, and the final inequality uses $\delta\geq2$.

Average \eqref{eq:local-increment} over uniform $m\in\cM(\alpha,\beta,\gamma)$. By \Cref{lem:uniform-extension},
\begin{align*}
  \E I_A&=F_{g_{\cF}}(\alpha+1,\beta,\gamma)
           -F_{g_{\cF}}(\alpha,\beta,\gamma),\\
  \E I_B&=F_{g_{\cF}}(\alpha,\beta+1,\gamma)
           -F_{g_{\cF}}(\alpha,\beta,\gamma),\\
  \E I_{AB}&=F_{g_{\cF}}(\alpha,\beta,\gamma+1)
           -F_{g_{\cF}}(\alpha,\beta,\gamma).
\end{align*}
Substitution into \eqref{eq:local-increment} yields \eqref{eq:H2}.
\end{proof}

\section{A doubling comparison for fixed-cardinality cuts}\label{sec:doubling}

We now apply the fixed-family interpolation to a two-block constraint. The disconnected endpoint produces two independent fixed-cardinality cut problems, while the unconditioned complete configuration model appearing on the other side of Huang's comparison admits the pointwise upper bound by maximum bisection. This gives the comparison that drives the rest of the proof.

For a multigraph $G$ on an $n$-vertex set and $0\leq k\leq n$, define
\begin{equation*}
  X_k(G)=\max_{\substack{S\subseteq V(G)\\ |S|=k}}\cut_G(S).
\end{equation*}
Taking complements gives
\begin{equation}\label{eq:X-complement}
  X_k(G)=X_{n-k}(G).
\end{equation}

If $\dd=(d_1,\dots,d_n)$ is a degree sequence, let $\dd^{\oplus2}$ be the degree sequence on two disjoint copies $A$ and $B$ of $[n]$ obtained by placing one copy of $\dd$ on each block.

\begin{proposition}\label{prop:doubling}
Let $\dd=(d_1,\dots,d_n)$ have even total degree
\[
  D=\sum_{i=1}^n d_i.
\]
Then, for every $0\leq k\leq n$,
\begin{equation}\label{eq:doubling-general}
  2\E X_k(\CM_{\dd})
  \leq
  \E\MB(\CM_{\dd^{\oplus2}})
  +\psi(D).
\end{equation}
\end{proposition}

\begin{proof}
Let $A$ and $B$ be the two copies of $[n]$. Consider the predetermined family
\begin{equation*}
  \cF_k
  =\{S\subseteq A\sqcup B:|S\cap A|=k,\ |S\cap B|=n-k\},
\end{equation*}
and write $g_k=g_{\cF_k}$. By \Cref{prop:fixed-family-huang}, the pseudo-parameter $g_k$ satisfies the hypotheses of \Cref{prop:huang}.

The total degree in each block is $D$, which is even. At the no-cross-edge endpoint in \eqref{eq:huang-cor}, every half-edge in $A$ is paired within $A$, and every half-edge in $B$ is paired within $B$. Indeed, the pairing space at this endpoint is the Cartesian product of the complete-pairing spaces in $A$ and $B$, so its uniform law factors. A uniform pairing at this endpoint is therefore the disjoint union of two independent configuration-model graphs $G_A,G_B$, each distributed as $\CM_{\dd}$. Since there are no cross-edges,
\begin{align*}
  g_k(G_A\sqcup G_B)
  &=\max_{\substack{S_A\subseteq A,\ |S_A|=k\\
                     S_B\subseteq B,\ |S_B|=n-k}}
      \bigl(\cut_{G_A}(S_A)+\cut_{G_B}(S_B)\bigr)\\
  &=X_k(G_A)+X_{n-k}(G_B).
\end{align*}
By \eqref{eq:X-complement}, the expectation of this quantity is $2\E X_k(\CM_{\dd})$.

For every complete pairing on $A\sqcup B$ and every $S\in\cF_k$,
\[
  |S|=k+(n-k)=n.
\]
Thus every feasible set is a bisection of the $2n$ vertices, and consequently
\[
  g_k(G)\leq\MB(G)
\]
for every multigraph $G$ on $A\sqcup B$.

The complete configuration model with degree sequence $\dd^{\oplus2}$ has total degree $2D$ and hence $D$ edges. Applying \eqref{eq:huang-cor} and then the preceding pointwise bound gives
\[
  2\E X_k(\CM_{\dd})
  \leq \E g_k(\CM_{\dd^{\oplus2}})+\psi(D)
  \leq \E\MB(\CM_{\dd^{\oplus2}})+\psi(D).
\]
This is \eqref{eq:doubling-general}.
\end{proof}

For the constant degree sequence, write
\begin{equation*}
  a_{n,k}=\E X_k(\CM_{n,d}),
  \qquad
  B_n=\E\MB(\CM_{n,d}),
  \qquad
  C_n=\E\MC(\CM_{n,d}),
\end{equation*}
whenever $n$ is even. Since the total degree of $\CM_{n,d}$ is $dn$, \Cref{prop:doubling} yields the estimate stated in the introduction.

\begin{corollary}\label{cor:doubling-regular}
For every fixed $d\geq1$, every even $n$, and every $0\leq k\leq n$,
\begin{equation*}
  2a_{n,k}
  \leq B_{2n}+7\sqrt{dn\log(1+dn)}.
\end{equation*}
\end{corollary}

\section{Concentration and the configuration-model limit}\label{sec:configuration-proof}

We next pass from the fixed-cardinality comparison to the unrestricted maximum cut. We first prove a concentration estimate that is uniform over every predetermined cut family. A union bound over the possible cardinalities then controls the maximum over $k$. Finally, Huang's convergence theorem for maximum bisection identifies the common limiting density.

\subsection{Uniform concentration in the cut cardinality}

\begin{lemma}[Pairing concentration]\label{lem:concentration}
Let $\dd$ be a degree sequence with even total degree $D$, and let $\cF\subseteq2^V$ be nonempty and predetermined. If $D\geq2$, then, for every $t\geq0$,
\begin{equation}\label{eq:concentration}
  \Pp\left(\left|g_{\cF}(\CM_{\dd})-
  \E g_{\cF}(\CM_{\dd})\right|\geq t\right)
  \leq 2\exp\left(-\frac{t^2}{4D}\right).
\end{equation}
For every even $D\geq0$,
\begin{equation}\label{eq:variance}
  \Var\bigl(g_{\cF}(\CM_{\dd})\bigr)\leq8D.
\end{equation}
\end{lemma}

\begin{proof}
If $D=0$, the graph and the value of $g_{\cF}$ are deterministic, so the variance statement is immediate. Assume henceforth that $D\geq2$. Expose a uniformly random complete pairing one pair at a time. At each step, take the least unmatched half-edge in a fixed ordering and reveal its uniformly random partner. There are $D/2$ steps. Let $(M_j)$ be the Doob martingale obtained by taking the conditional expectation of $g_{\cF}$ after the first $j$ pairs have been exposed.

We claim that $|M_j-M_{j-1}|\leq2$ for every $j$. Fix a stage, let $h$ be the half-edge currently being paired, and consider two possible partners $x$ and $y$. There is a measure-preserving bijection between the completion space conditional on choosing $hx$ and the completion space conditional on choosing $hy$. Indeed, in a completion containing the pair $\{h,x\}$, let $z$ be the partner of $y$ and perform the switching
\[
  \{h,x\},\ \{y,z\}
  \quad\longmapsto\quad
  \{h,y\},\ \{x,z\}.
\]
All other pairs are left unchanged. The inverse map is obtained by applying the same switching with $x$ and $y$ interchanged. Hence the bijection preserves the uniform laws on the two completion spaces.

For a fixed $S\in\cF$, write
\[
  \chi_S(u,v)=\1_{\{|\{u,v\}\cap S|=1\}}
\]
for the indicator that the edge joining the vertices incident to the half-edges $u$ and $v$ crosses the cut defined by $S$. The contribution of the two affected pairs before the switching is
\[
  \chi_S(h,x)+\chi_S(y,z)\in\{0,1,2\},
\]
whereas their contribution afterward is
\[
  \chi_S(h,y)+\chi_S(x,z)\in\{0,1,2\}.
\]
Their difference in absolute value is therefore at most $2$. All other edge contributions agree, so the two values of $\cut(S)$ differ by at most $2$. Taking the maximum over the same fixed family $\cF$ preserves this bound.

It follows that the conditional expectations corresponding to any two possible partners of $h$ differ by at most $2$. Since $M_{j-1}$ is a convex combination of these conditional expectations and $M_j$ is one of them, we obtain
\[
  |M_j-M_{j-1}|\leq2,
\]
as claimed.

Azuma--Hoeffding's inequality now gives
\[
  \Pp\bigl(|M_{D/2}-M_0|\geq t\bigr)
  \leq2\exp\left(-\frac{t^2}{2\cdot(D/2)\cdot2^2}\right)
  =2\exp\left(-\frac{t^2}{4D}\right).
\]
This is \eqref{eq:concentration}. Integrating the tail bound yields
\begin{align*}
  \Var(g_{\cF}(\CM_{\dd}))
  &=\int_0^\infty 2t\,
    \Pp\left(\left|g_{\cF}(\CM_{\dd})-
    \E g_{\cF}(\CM_{\dd})\right|\geq t\right)\,dt\\
  &\leq\int_0^\infty4t\exp\left(-\frac{t^2}{4D}\right)dt
  =8D,
\end{align*}
which proves \eqref{eq:variance}.
\end{proof}

\begin{lemma}\label{lem:max-over-k}
For fixed $d\geq1$ and even $n$,
\begin{equation}\label{eq:max-over-k}
  C_n
  \leq
  \max_{0\leq k\leq n}a_{n,k}
  +r_{n,d},
\end{equation}
where
\begin{equation*}
  r_{n,d}
  =2\sqrt{dn\log(2n+2)}
   +\sqrt{\frac{dn}{\log(2n+2)}}.
\end{equation*}
In particular, $r_{n,d}=O_d(\sqrt{n\log n})$.
\end{lemma}

\begin{proof}
Since $\MC(G)=\max_{0\leq k\leq n}X_k(G)$,
\[
  C_n
  \leq \max_k a_{n,k}
  +\E\left[\max_k\bigl(X_k(\CM_{n,d})-a_{n,k}\bigr)\right].
\]
For each $k$, \Cref{lem:concentration} applies to the family of all $k$-element subsets. A union bound therefore gives
\begin{equation}\label{eq:union-tail}
  \Pp\left(\max_k\bigl(X_k(\CM_{n,d})-a_{n,k}\bigr)\geq t\right)
  \leq2(n+1)\exp\left(-\frac{t^2}{4dn}\right).
\end{equation}
Set
\[
  t_0=2\sqrt{dn\log(2n+2)}.
\]
Using $\E Y\leq\E Y_+$ and integrating \eqref{eq:union-tail}, we obtain
\[
  \E\max_k(X_k-a_{n,k})
  \leq t_0+
  \int_{t_0}^\infty2(n+1)
  \exp\left(-\frac{t^2}{4dn}\right)dt.
\]
For $x>0$,
\[
  \int_x^\infty\exp\left(-\frac{t^2}{4dn}\right)dt
  \leq \frac{2dn}{x}\exp\left(-\frac{x^2}{4dn}\right),
\]
which follows by comparing $1$ with $t/x$ inside the integral. Since
\[
  \exp\left(-\frac{t_0^2}{4dn}\right)=\frac1{2n+2},
\]
the tail integral is at most $2dn/t_0$. Consequently,
\[
  \E\max_k(X_k-a_{n,k})
  \leq t_0+\frac{2dn}{t_0}=r_{n,d}.
\]
This proves \eqref{eq:max-over-k}.
\end{proof}

Combining \Cref{cor:doubling-regular,lem:max-over-k} gives
\begin{equation}\label{eq:C-vs-B2n}
  C_n
  \leq
  \frac12B_{2n}
  +\frac72\sqrt{dn\log(1+dn)}
  +r_{n,d}.
\end{equation}

\subsection{The common limiting density}

We now invoke Huang's convergence theorem for maximum bisection.

\begin{lemma}\label{lem:b-limit}
For each fixed integer $d\geq1$, there is a constant $b_d$ such that
\begin{equation}\label{eq:B-limit}
  \frac{B_n}{n}\longrightarrow b_d
\end{equation}
as $n\to\infty$ through even integers.
\end{lemma}

\begin{proof}
Huang's theorem \cite[Theorem~3.1]{Huang2018} applies to any sequence of degree lists whose empirical degree distributions converge and whose empirical means converge to the finite mean of the limiting distribution.

To accommodate the parity condition, define a degree list $\dd^{(N)}$ for every $N\geq1$. If $Nd$ is even, let every degree be $d$. If $Nd$ is odd, then both $N$ and $d$ are odd; in this case, give one vertex degree $d+1$ and every other vertex degree $d$. The total degree is always even. Moreover, the empirical degree distribution converges to the point mass at $d$, and the empirical mean converges to $d$. Huang's theorem therefore gives convergence of
\[
  \frac1N\E\MB(\CM_{\dd^{(N)}}).
\]
Along even $N$, the degree list is exactly constant and equal to $d$, so this subsequence is $B_N/N$. This proves \eqref{eq:B-limit}.
\end{proof}

\begin{theorem}\label{thm:configuration}
For every fixed integer $d\geq1$, as $n\to\infty$ through even integers,
\begin{equation}\label{eq:configuration-result}
  C_n-B_n=o(n).
\end{equation}
Moreover,
\[
  \frac{C_n}{n}\longrightarrow b_d,
  \qquad
  \frac{B_n}{n}\longrightarrow b_d.
\]
\end{theorem}

\begin{proof}
Divide \eqref{eq:C-vs-B2n} by $n$. Both error terms are $o(n)$, and \Cref{lem:b-limit} gives
\[
  \limsup_{\substack{n\to\infty\\ n\text{ even}}}\frac{C_n}{n}
  \leq
  \lim_{\substack{n\to\infty\\ n\text{ even}}}\frac{B_{2n}}{2n}
  =b_d.
\]
On the other hand, $\MC(G)\geq\MB(G)$ for every graph, and hence $C_n\geq B_n$. Therefore
\[
  \liminf_{\substack{n\to\infty\\ n\text{ even}}}\frac{C_n}{n}
  \geq
  \lim_{\substack{n\to\infty\\ n\text{ even}}}\frac{B_n}{n}
  =b_d.
\]
It follows that $C_n/n\to b_d$. Together with \eqref{eq:B-limit}, this proves \eqref{eq:configuration-result}.
\end{proof}

\section{Transfer to simple regular graphs}\label{sec:simple}

We finish by conditioning the configuration model on simplicity. Janson's criterion gives a uniform positive lower bound for the simplicity probability at fixed degree. The concentration estimate from \Cref{sec:configuration-proof} then shows that conditioning changes the expected maximum cut and expected maximum bisection by only $O_d(\sqrt n)$.

Let $\mathsf S_{n,d}$ be the event that $\CM_{n,d}$ is simple. Janson's simplicity criterion \cite[Theorem~1.1]{Janson2014} states that, for a sequence of configuration models whose total degrees tend to infinity,
\[
  \liminf\Pp(\text{simple})>0
  \quad\Longleftrightarrow\quad
  \sum_i d_i^2=O\left(\sum_i d_i\right).
\]
For the $d$-regular degree sequence,
\[
  \sum_{i=1}^n d_i^2=nd^2=d\sum_{i=1}^n d_i.
\]
Thus there are constants $c_d>0$ and $n_0(d)$ such that
\begin{equation}\label{eq:simple-prob}
  \Pp(\mathsf S_{n,d})\geq c_d
\end{equation}
for every even $n\geq n_0(d)$.

Conditioned on $\mathsf S_{n,d}$, the configuration model is uniform over the simple $d$-regular graphs on $[n]$. Indeed, every such graph arises from exactly $(d!)^n$ assignments of its incident edges to the labeled half-edges at the vertices. Consequently,
\begin{equation}\label{eq:conditional-law}
  \mathcal L(\CM_{n,d}\mid\mathsf S_{n,d})
  =\mathcal L(G_{n,d}).
\end{equation}

\begin{lemma}\label{lem:conditioning}
For $Z=\MC(\CM_{n,d})$ or $Z=\MB(\CM_{n,d})$,
\begin{equation}\label{eq:conditioning-bound}
  \left|\E[Z\mid\mathsf S_{n,d}]-\E Z\right|
  =O_d(\sqrt n).
\end{equation}
\end{lemma}

\begin{proof}
Both $\MC$ and $\MB$ are of the form $g_{\cF}$ for a predetermined family of cuts. Hence \eqref{eq:variance} gives
\[
  \Var(Z)\leq8dn.
\]
Using \eqref{eq:simple-prob} and Cauchy--Schwarz,
\begin{align*}
  \left|\E[Z\mid\mathsf S_{n,d}]-\E Z\right|
  &=\frac{\left|\E\bigl[(Z-\E Z)\1_{\mathsf S_{n,d}}\bigr]\right|}
          {\Pp(\mathsf S_{n,d})}\\
  &\leq
  \frac{\sqrt{\Var(Z)\Pp(\mathsf S_{n,d})}}
       {\Pp(\mathsf S_{n,d})}\\
  &\leq\sqrt{\frac{8dn}{c_d}}.
\end{align*}
This proves \eqref{eq:conditioning-bound}.
\end{proof}

We can now prove the main theorem.

\begin{proof}[Proof of \Cref{thm:main}]
By \eqref{eq:conditional-law} and \Cref{lem:conditioning},
\begin{align*}
  \E\MC(G_{n,d})&=C_n+O_d(\sqrt n),\\
  \E\MB(G_{n,d})&=B_n+O_d(\sqrt n).
\end{align*}
Subtracting and applying \Cref{thm:configuration}, we obtain
\[
  \E\MC(G_{n,d})-\E\MB(G_{n,d})
  =C_n-B_n+O_d(\sqrt n)=o(n).
\]
The convergence of both normalized expectations to $b_d$ follows from \Cref{thm:configuration} and the same conditioning estimates.
\end{proof}

\begin{corollary}
For fixed $d\geq1$,
\[
  \frac{\MC(G_{n,d})-\MB(G_{n,d})}{n}
  \longrightarrow0
\]
in $L^1$ and hence in probability as $n\to\infty$ through even integers $n>d$.
\end{corollary}

\begin{proof}
The difference is nonnegative, and its expectation divided by $n$ tends to zero by \Cref{thm:main}.
\end{proof}

\bibliographystyle{amsplain}
\bibliography{maxcut_maxbis_random_regular_references}

\end{document}